\documentclass[12pt, a4paper]{article}

\usepackage{amssymb}
\usepackage{amstext}
\usepackage{amsthm}
\usepackage{amsmath}
\usepackage{amscd}
\usepackage{latexsym}
\usepackage{amsfonts}
\usepackage[dvips]{graphicx}
\usepackage{color}

\newtheorem{theorem}{Theorem}[section]
\newtheorem{proposition}[theorem]{Proposition}
\newtheorem{lemma}[theorem]{Lemma}
\newtheorem{corollary}[theorem]{Corollary}

\newtheorem{problem}[theorem]{Problem}
\theoremstyle{definition}
\newtheorem{definition}[theorem]{Definition}
\newtheorem{example}[theorem]{Example}

\theoremstyle{remark}
\newtheorem{remark}[theorem]{Remark}
\numberwithin{equation}{section}

\def\N{\mathbb N}
\def\Z{\mathbb Z}
\def\Mon{\mathrm{Mon}}
\def\Kdim{\mathrm{dim}}
\def\lt{{\ell t}}
\def\lm{{\ell m}}
\def\LCM{ LCM }
\def\lc{{\ell c}}
\def\G{{\mathcal G}}
\def\SS{{\mathcal T}}
\def\Sn{S^{\langle n \rangle}}
\def\Sm{S^{\langle m \rangle}}

\title{\textbf{Gr\"{o}bner bases for the polynomial ring with infinite variables and their applications}}
\author{Kei-ichiro Iima and Yuji Yoshino}
\date{}

\begin{document}

\maketitle

\begin{abstract}
We develop the theory of Gr\"obner bases for ideals in a polynomial ring with countably infinite variables over a field. 
As an application we reconstruct some of the one-one correspondences among various sets of partitions by using division algorithm. 
\end{abstract}


\section{Introduction} 
The purpose of this paper is to develop the theory of Gr\"obner bases for ideals in a polynomial ring  $k[x_1, x_2, \ldots]$  with countably infinite variables over a field $k$. 
In such a case, ideals are not necessarily finitely generated, and hence the Gr\"obner bases for ideals might be consisting of infinite polynomials. 
However we shall claim that there is still an algorithm to get the Gr\"obner base for a given ideal. 

This idea of Gr\"obner bases for infinitely generated ideals is strongly motivated by the following observation. 
Recall that a sequence  $\lambda = (\lambda _1, \lambda _2 , \ldots , \lambda _r)$ of positive integers is called  a partition of a non-negative integer $n$  if  the equality $\lambda _1 + \lambda _2 + \cdots + \lambda _r = n$ holds and $\lambda _1 \geq \lambda _2 \geq \cdots \geq \lambda _r \geq 1$.
In such a case we denote it by $\lambda \vdash n$.
We are concerned with the following sets of partitions:
$$
\begin{array}{rll} 
A(n) &= \{ \ \ \lambda \vdash n &| \ \ \lambda _i \equiv \pm 1\pmod 6 \ \}, \vspace{6pt}  \\
B(n) &= \{ \ \ \lambda \vdash n &| \ \ \lambda _i \equiv \pm 1\pmod 3 , \quad \lambda _1 > \lambda _2 > \cdots > \lambda _r \ \}, \vspace{6pt}\\
C(n) &= \{ \ \ \lambda \vdash n &| \ \ \text{each $\lambda _i$ is odd, and} \\ 
 & &\ \ \ \text{any number appears in $\lambda _i$'s at most two times \}}. \vspace{6pt}
\end{array}
$$
It is known by the famous Schur's equalities (see \cite{A}) that all these sets  $A(n)$, $B(n)$ and $C(n)$ have the same cardinality for all $n \in \N$. 
It is also known that the one-to-one correspondences among these three sets are realized in some combinatorial way using 2-adic or 3-adic expansions of integers. 
However such one-to-one correspondences can be reconstructed through the division algorithm by using the theory of Gr\"{o}bner bases. 
For this, we need to extend the theory of Gr\"{o}bner bases to a polynomial ring with infinitely many variables.

In Section 1,  we shall give necessary definitions of initial ideals, Gr\"obner bases,  S-polynomials and regular sequences  in  the polynomial ring $k[x_1, x_2, \ldots]$. 
And we develop the theory of Gr\"obner bases for ideals in such a polynomial ring by presenting a sequence of propositions, 
most of which goes in parallel with the ordinary case for ideals in polynomial rings with finitely many variables. 
But the difference is that ideals are not necessarily finitely generated and we need to argue about infinite set of polynomials as Gr\"obner bases and regular sequences. 
One of the essentially new results in this paper is Theorem \ref{construction reduced base} where we give an algorithm  to get the reduced Gr\"obner bases. 
The other one is Theorem \ref{5} in which we show that any permutation of a homogeneous regular sequence of infinite length is again a regular sequence.

In Section 2,  we apply the theory developed in Section 1 to the sets of partitions. 
The main result is Theorem \ref{main theorem}, where we give  one-to-one correspondences between various sets of partitions by using division algorithm in the theory of Gr\"obner bases. 
As one of the applications we shall give the bijective mapping among the above mentioned sets $A(n)$, $B(n)$  and  $C(n)$.

\subsection{Gr\"{o}bner bases for ideals}

Throughout this paper, let $k$  be any field and let $S = k[x_1,x_2,\ldots]$ be a polynomial ring with countably infinite variables.
We denote by  $\Z _{\ge 0}^{(\infty)}$ 
 the set of all sequences $a = (a_1,a_2,\ldots)$ of integers where $a_i=0$  for all $i$ but finite number of integers. 
Also we denote by $\Mon (S)$  the set of all monomials in  $S$. 
Since any monomial is described uniquely as $x^a = \prod_{i}x_{i}^{a_{i}}$ for some $a =  (a_1,a_2,\ldots) \in \Z_{\ge 0}^{(\infty)}$, 
we can identify these sets, i.e. $\Mon (S) \ \cong \ \Z _{\ge 0}^{(\infty)}$. 
If we attach degree on  $S$  by $\deg x_i =d_i$, then a monomial $x^a$  has degree  $\deg x ^a = \sum _{i=1}^{\infty} a_i d_i$. 
In the rest of the paper,  we assume that the degrees $d_i$'s are chosen in such a way that there are only a finite number of monomials of degree $d$ for each  $d \in \Bbb N$.
For example, the simplest way of attaching degree is that $\deg x_i = i$  for all $i \in \N$.

\begin{definition}
A total order $>$  on  $\Mon (S)$ is called a monomial order
 if $(\Mon (S), \ >)$ is a well-ordered set, and it is compatible with the multiplication of monomials, i.e. 
$x^a > x^b$ implies $x^cx^a > x^cx^b$ 
for all $x^a, x^b, x^c \in \Mon (S)$.
(See \cite[Chapter 15]{E}.)
\end{definition}

Note that the order $x_1 > x_2 > x_3 > \cdots$  is not acceptable for monomial order,  since it violates the well-ordering condition. 
On the other hand, if we are given any monomial order  $>$,  then, renumbering the variables, we may assume that  $x_1 < x_2 < x_3 < \cdots$.

The following are examples of monomial orders on $\Mon (S)$. 
(See \cite[Chapter 15, pp. 329--330]{E}.)

\begin{example}\label{order example} 
Let $a = (a_1, a_2, \ldots)$ and $b = (b_1, b_2, \ldots)$ be elements in $\Z_{\ge 0}^{(\infty)}$. 

\begin{itemize}
\item[(1)]
The pure lexicographic order $>_{pl}$  is defined in such a way that $x^a >_{pl} x^b$  if and only if $a_i > b_i$ for the last index $i$ with $a_i \neq b_i$.

\item[(2)]
The homogeneous  (resp. anti-) lexicographic order  $>_{hl}$  (resp. $>_{hal}$) is defined in such a way that  $x^a >_{hl} x^b$   (resp. $x^a >_{hal}  x^b$)   if and only if either  $\deg x^a > \deg x^b$ or $\deg x^a = \deg x^b$ and $a_i > b_i$ for the last (resp. first) index $i$ with $a_i \neq b_i$.

\item[(3)]
The homogeneous (resp. anti-) reverse lexicographic order 
$>_{hrl}$ (resp. $>_{harl}$) is defined as follows:
$x^a >_{hrl} x^b$  (resp.  $x^a >_{harl} x^b$)  if and only if either $\deg x^a > \deg x^b$ or $\deg x^a = \deg x^b$ and $a_i < b_i$ for the first (resp. last) index $i$ with $a_i \neq b_i$. 
\end{itemize}
\end{example}

As in the orders in (2) to (3), if it satisfies that 
 $\deg x^a > \deg x^b$  implies  $x^a > x^b$, then we say that the order  $>$  is homogeneous. 
The monomial orders in Example \ref{order example} are all distinct as shown in the following example in which  $\deg x_i = i$  for  $i \in \N$:  

$$
\begin{array}{clclclclcl}
x_4  &>_{hl}  &x_1x_3  &>_{hl}  &x_2^2 &>_{hl}  &x_1^2x_2  &>_{hl} &x_1^4, \vspace{4pt} \\
\vspace{4pt}
x_1^4 &>_{hal} &x_1^2x_2 &>_{hal} &x_1x_3 &>_{hal} &x_2^2 &>_{hal} &x_4 , \\
\vspace{4pt}
x_4   &>_{hrl}  &x_2^2 &>_{hrl} &x_1x_3 &>_{hrl} &x_1^2x_2 &>_{hrl} &x_1^4 , \\
\vspace{4pt}
x_1^4 &>_{harl} &x_1^2x_2 &>_{harl} &x_2^2 &>_{harl} &x_1x_3 &>_{harl} &x_4 .
\end{array}
$$

Now suppose that a monomial order $>$  on  $\Mon (S)$ is given and we fix it in the rest of this section.  
Then,  any non-zero polynomial  $f \in S$  is expressed as  
$$
f = c_1 x^{a{(1)}} + c_{2} x^{a{(2)}}+ \cdots + c_{r} x^{a{(r)}}, 
$$
 where  $c_{i} \neq 0 \in k$  and  $x^{a{(1)}} > x^{a{(2)}} > \ldots > x^{a{(r)}}$. 
In such a case,  
the leading term, the leading monomial and the leading coefficient of $f$  are given respectively as 
$\lt  (f) = c_{1}x^{a{(1)}}$,  $\lm (f) = x^{a{(1)}}$ and 
$\lc (f) = c_{1}$. 
For a non-zero ideal $I \subset S$, the initial ideal  $in(I)$  of  $I$  is defined to be the ideal generated by all the leading terms $\lt  (f)$  of non-zero polynomials  $f \in I$. 
(See \cite[Chapter 15, p. 329]{E}.)

For a positive integer $n$,  we set  $\Sn = k[x_1, x_2, \ldots , x_n]$ which is a polynomial subring of $S$. 
Note that there is a filtration  $S^{\langle 1 \rangle} \subset \cdots \subset S^{\langle n \rangle} \subset S^{\langle n+1 \rangle} \subset \cdots \subset S$  and $S = \cup _{n=1} ^{\infty} \Sn$. 
Since $\Mon (\Sn)  \subset  \Mon (S)$, we always employ the restricted monomial order from  $\Mon (S)$  as a monomial order on $\Mon (\Sn)$. 
Therefore if  $f \in S$ then the leading monomial of  $f$  in  $\Sn$  is independent of  any  such $n$  with  $f \in \Sn$.

One easily observes the following remark. 
(See \cite[Chapter 15, Proposition 15.4]{E}.)

\begin{remark}\label{characterize order}
\begin{itemize}
\item[$(1)$]
Assume that the monomial order is a pure lexicographic order. 
If $\lm (f) \in \Sn$ for $f \in S$  and  $n \in \N$,  then $f \in \Sn$. 

\item[$(2)$]
Assume that the monomial order  is  a homogeneous lexicographic order. 
If  $\lm (f) \in \Sn$  for a homogeneous polynomial  $f \in S$  and  $n \in \Bbb N$, then $f \in \Sn$. 

\item[$(3)$]
Assume that the monomial order is a homogeneous reverse lexicographic order. 
If  $\lm (f) \in (x_1, x_2, \ldots , x_n)S$ for a homogeneous polynomial $f \in S$ and $n \in \Bbb N$, then  $f \in (x_1, x_2, \ldots , x_n)S$.
\end{itemize}
\end{remark}

\bigskip

The Gr\"obner base for an ideal of  $S$  is defined similarly to the ordinary case. 
(See \cite[Chapter 15]{E}.)

\begin{definition}
A subset $\G$  of an ideal $I$ of  $S$  is called a  Gr\"obner base for $I$ if $\{ \ \lm (g) \ | \ g \in \G \}$  generates the initial ideal  $in (I)$. 
\end{definition}

It is easily observed that a Gr\"obner base for $I$ is actually a generating set of $I$. 
Note that an ideal $I$ does not necessarily admit a finite Gr\"obner base, since  $S$  is not a Noetherian ring. 
However the generating set of  $in (I)$  is a subset of $\Mon (S)$ which must be a  countable set, hence 
one can always take a countable set of polynomials as a Gr\"obner base for $I$.

Any argument concerning Gr\"obner bases for an ideal of  $S$  can be reduced to the ordinary case for the polynomial rings with finite variables by the following lemma, in which, for a subset $\G$  of  $S$, we denote by $in (\G)$  the set of all the leading monomials  $\lm (g)$  for  $g \in \G$.

\begin{lemma}\label{NSC for Groebner}
The following conditions are equivalent for a subset $\G$  of  an ideal $I$  of  $S$.  

\begin{itemize}
\item[$(1)$]
$\G$  is a Gr\"obner base for  $I$. 

\item[$(2)$]
$in (\G) \cap \Sn$   generates the initial ideal  $in (I \cap \Sn)$ 
for all integers $n$. 

\item[$(3)$]
$in (\G) \cap \Sn$   generates the initial ideal  $in (I \cap \Sn)$ 
for infinitely many integers $n$. 
\end{itemize}
\end{lemma}

\begin{proof}
$(1) \Rightarrow (2)$:  
Suppose $\G$  is a Gr\"obner base for $I$ and let  $f \in I \cap \Sn$. 
Then there is $g \in \G$ such that  $\lm (g)$  divides  $\lm (f)$. 
Since  $\lm (f) \in \Sn$, we have  $\lm (g) \in  in (\G) \cap \Sn$. 
Thus $in (\G) \cap \Sn$   generates $in (I \cap \Sn)$.  

$(2) \Rightarrow (3)$:  
Trivial.  

$(3) \Rightarrow (1)$:  
Let  $f \in I$  be any element. 
Take an integer $n$ so that  $f \in \Sn$. 
Then, by the condition  (3),  there is an integer $m \geq n$  such that 
 $in (I \cap \Sm)$  is generated by $in (\G) \cap \Sm$. 
Since  $f \in I \cap \Sm$, the leading monomial  $\lm (f)$ is a multiple of $\lm (g)$  for some  $g \in \G$.
Therefore  $\G$  is a Gr\"obner base for  $I$. 
\end{proof}

\begin{corollary}\label{cor 1}
Let  $\G$  be a subset of an ideal $I$ of  $S$. 
Assume that  $\G \cap \Sn$  is a Gr\"obner base for an ideal $I \cap \Sn$
for infinitely many integers  $n$. 
Then  $\G$  is a Gr\"obner base for  $I$. 
\end{corollary}

\begin{proof}
It follows from the definition that 
$in (\G \cap \Sn) \subseteq in (\G) \cap \Sn$. 
Since  $in (\G \cap \Sn)$ generates the initial ideal $in (I \cap \Sn)$ for such infinitely many integers  $n$  in the assumption,  $\G$  is a Gr\"obner base for $I$ by Lemma \ref{NSC for Groebner}.
\end{proof}

Note that  the inclusion  $in (\G \cap \Sn) \subseteq in (\G) \cap \Sn$ is strict in general. 

\begin{corollary}\label{cor 2}
Assume that the monomial order is a pure lexicographic order. 
If $\G$  is  a Gr\"obner base for an ideal  $I$  of  $S$, then 
 $\G \cap \Sn$ is a Gr\"obner base for $I \cap \Sn$  for all $n \in \N$. 
\end{corollary}

\begin{proof}
It follows from Remark \ref{characterize order}(1) that  
the equality $in (\G \cap \Sn) = in (\G) \cap \Sn$ holds in this case 
and it is a generating  set of  $in (I \cap \Sn)$ for each $n$. 
\end{proof}

Now we can construct a Gr\"obner base for any ideal of  $S$.

\begin{proposition}\label{construct groebner}
Let  $I$ be an ideal of  $S$ and let  $C$  be an arbitrary infinite subset of  $\N$. 
For each  $n \in C$, take a Gr\"obner base  $\G_n$  for an ideal $I \cap \Sn$  inside  $\Sn$. 
Then, the set  $\bigcup _{n \in C} \G _n$  is a Gr\"obner base for $I$.
\end{proposition}

\begin{proof}
Set  $\G = \bigcup _{n \in C} \G _n$, and we see 
that  $\G \cap \Sn$ contains  $\G _n$  for each  $n \in C$, 
hence  $\G \cap \Sn$, as well as  $\G _n$,  is a Gr\"obner base for  $I \cap \Sn$ for such  $n$. 
Hence  Corollary \ref{cor 1}  implies that  $\G$  is a Gr\"obner base for $I$.  \end{proof}

Compare the following division algorithm with that in \cite[Chapter 15, Proposition-Definition 15.6]{E}.

\begin{proposition}[Division  algorithm]\label{division} 
Let $\G$  be a subset of  $S$. 
Then any non-zero polynomial $f \in S$  has an expression
$$
f = f_1g_1 + f_2 g_2 + \cdots + f_s g_s + f', 
$$
with  $g_i \in \G$  and  $f_i, f' \in S$  such that the following conditions hold:
\begin{itemize}
\item[$(1)$]
If we write  $f' = \sum _{i=1}^t  c_i x^{a{(i)}}$  with  $c_i \not= 0 \in k$,  then  $x^{a{(i)}} \notin in (\G S)$ for each $i = 1, 2, \ldots , t$.  

\item[$(2)$]
If  $f_i g_{i} \neq 0$, then $\lm (f_i g_i) \leq \lm (f)$.
\end{itemize}

Any such $f'$ is called a remainder of  $f$  with respect to $\G$. 
Note that a remainder is in general not necessarily unique.  
But if $\G$  is a Gr\"obner base for $I= \G S$, then 
 a remainder of $f$ with respect to $\G$ is uniquely determined. 
\end{proposition}

\begin{proof}
The existence of such an expression is proved by induction on  $\lm (f)$, 
which goes in a similar way to the proof in \cite[Proposition 15.8]{E}. 
In fact, if  $\lm (f) \in in (\G S)$, then one can find  $g \in \G$  whose leading monomial divides $\lm (f)$, i.e.  $\lt (f) = c \mu \cdot \lt (g)$  for some monomial $\mu$ and $c \in k$. 
In this case, since  $\lm (f - c \mu g) < \lm (f)$, the proof is done by the induction hypothesis. 
If  $\lm (f) \not\in in (\G S)$, then arguing about  $f - \lt (f)$ we will have a desired expression again by the induction hypothesis. 

The last half of the proposition is obvious from the definition of Gr\"obner bases.
\end{proof}

Let us assume that $S$  is a graded ring with homogeneous monomial order and  $\G$  is a set of homogeneous polynomials. 
Then we remark that  if  $f \in S$  is a homogenous polynomial,  then  
all polynomials in the expression in Proposition \ref{division} can be taken to be homogeneous, hence the remainder of $f$  with respect to  $\G$  is also homogenous.

By virtue of Proposition \ref{division}, the membership problem has a solution 
as in the ordinary cases. 

\begin{corollary}\label{membership}
Let  $\G$  be a Gr\"obner base for an ideal  $I$. 
Then an element  $f \in S$  belongs to  $I$ if and only if $0$  is the remainder of $f$ with respect to $\G$.
\end{corollary}

Recall that a Gr\"obner base  $\G$ for a non-zero ideal $I$ of  $S$  is called  a reduced  Gr\"obner base if every $g \in \G$  is a monic polynomial, i.e.  $\lc (g) = 1$, and  $\lm (g)$ does not divide any term of $h$ for any $g \not= h \in \G$. 
Note that  any ideal of  $\Sn$ has a  unique reduced Gr\"obner base. 
(See \cite[Chapter 1]{S}.)

\begin{proposition}\label{existence reduced base}
For an arbitrary non-zero ideal  $I$  of  $S$, there uniquely exists a reduced Gr\"obner base for $I$. 
\end{proposition}

\begin{proof}
Let  $\{ \mu _{\lambda} | \ \lambda \in \Lambda \}$  be a minimal generating set of the monomial ideal $in (I)$, i.e. 
it generates $in (I)$ and any monomial dividing properly $\mu _{\lambda}$  does not belong to $in (I)$. 
It is easy to see that such a minimal generating set uniquely exists for  $in (I)$. 
Take $g_{\lambda} \in I$  such that  $\lm (g_{\lambda}) = \mu _{\lambda}$, and set  $\G = \{ g_{\lambda} |\ \lambda \in \Lambda\}$. 
Then it is clear that  $\G$  is a Gr\"obner base for $I$. 
Replacing  $g_{\lambda}$  with its remainder with respect $\G \backslash \{ g_{\lambda} \}$, we can see that  $\G$  is a reduced Gr\"obner base for $I$. 

To prove the uniqueness, let  $\G$  and $\G'$  be reduced Gr\"obner bases for $I$. 
Assume that $\G \not\subset \G'$. 
Then take  $g \in \G \backslash \G'$  so that $\lm (g)$  is minimum among those polynomials in  $\G \backslash \G'$. 
Since  $\G'$ is a Gr\"obner base for $I$, there is $g'\in \G'$ such that  $\lm (g')$  divides $\lm (g)$. 
Then it forces $\lm (g) = \lm (g')$, since  $g$  is an element of a reduced Gr\"obner base. 
Note that every term of the polynomial  $g -g'$  is not belonging to $in (\G S)$,  since it is smaller than  $\lm (g)$ and since any monomial $\mu \in \G$  with  $\mu < \lm (g)$ belongs to  $\G'$. 
As a result, we see that  $g-g'$  is a remainder of  $g-g'$ itself  with respect to  $\G$. 
As we remarked above as the membership problem, this forces  $g =g'$  and hence  $g \in \G '$. 
This contradiction shows that $\G \subseteq \G'$. 
And by the symmetry  of arguments  we conclude that $\G = \G'$. 
\end{proof}

\begin{theorem}\label{construction reduced base}
Let  $I$  be a non-zero ideal  $I$  of  $S$. 
Take a reduced Gr\"obner base $\G _n$ for $I \cap \Sn$ inside the polynomial ring  $\Sn$  for each $n$, and consider the following set of polynomials in  $S$;
$$
\overline{\G} = \bigcup _{m=1}^{\infty} \left( \bigcap _{n=m}^{\infty} \G _{n} \right).
$$
Then  $\overline{\G}$  is a reduced Gr\"obner base for  $I$. 
\end{theorem}

\begin{proof}
Let  $\G$  be the unique reduced Gr\"obner base for $I$, whose existence  we have shown in Proposition \ref{existence reduced base}. 
First we prove that  $\G \subseteq \overline{\G}$. 
To show this, let  $g \in \G$ and take an integer  $n$  so that  $g \in \Sn$.
Note that  $\lm (g)$ is not divisible by any other monomial belonging to the minimal generating set of  the monomial ideal  $in (I)$, 
and also note that any terms of  $g$  other than  $\lt (g)$  are not divisible by any monomial in the minimal generating set of  $in (I)$. 
This implies that  $g$  is a member of the reduced Gr\"obner base for  $I \cap \Sn$, hence  $g \in \G_n$ for such $n$. 
Therefore  $g \in \overline{\G}$,  and we have shown  $\G \subseteq \overline{\G}$. 
Since  $\overline{\G}$  contains a Gr\"obner base for  $I$  and  since $\overline{\G} \subseteq I$,  $\overline{\G}$   is a Gr\"obner base for $I$ as well.

To show that  $\overline{\G}$  is a reduced Gr\"obner base for $I$,  let  $g, h$  be distinct elements of  $\overline{\G}$. 
Take an integer $m$  so that  $g, h \in \bigcap _{n \geq m} \G _n$, in particular  $g, h \in \G_m$. 
Since  $\G _m$  is a reduced Gr\"obner base, we see that $\lm (h)$  does not divide  any term of  $g$. 
Hence the Gr\"obner base  $\overline{\G}$  is a reduced one. 
\end{proof}

Recall that the $S$-polynomial of elements  $f, g \in S$ are defined to be 
$$
S(f,g) = \frac{ \LCM  \{ \lm (f) , \ \lm (g) \}}{\lt (f)} \ f - \frac{ \LCM  \{ \lm (f) , \ \lm (g) \}}{\lt (g)} \ g.
$$
(See \cite[Chapter 15.4]{E}.)
Note here that  $S$  is a unique factorization domain and that the least common multiple  $\LCM$  is defined well. 
Now the Buchberger's criterion for Gr\"obner bases is proved in a similar way to  ordinary cases for polynomial rings with finite variables \cite[Theorem 15.8]{E}.

\begin{proposition}[Buchberger's  criterion]\label{buchberger}
\indent Let $\G$ be a generating subset of an ideal $I \subset S$. 
Then $\G$ is a Gr\"obner base for  $I$ if and only if $0$  is a remainder of   $S(g_{\lambda},g_{\mu})$  with respect to $\G$  for all pairs  $(g_{\lambda},   g_{\mu})$  of elements in  $\G$. 
\end{proposition}

\begin{proof}
The \lq\lq only if\rq\rq part is obvious. 
To prove the \lq\lq if\rq\rq    part,  let  $f \in I$  and we show that $\lm (f)$  is a multiple of  $\lm (g)$  for an element $g \in \G$. 
Since  $\G$  generates the ideal  $I$,  an equality $f = \sum _{i=1} ^r h_i g_i$  holds for some $g_i \in \G$  and  $h_i \in S \ (1\leq i \leq r)$. 
Let  $\mu$  be the monomial which is maximum  among  $\lm (h_ig_i) \ (1\leq i \leq r)$. 
If  $\lm (f) = \mu$, then there is nothing to prove, because  $\lm (f) = \lm (h_ig_i) = \lm (h_i) \lm (g_i)$ for some $i$. 
If  $\lm (f) < \mu$, then applying the following claims to  
$f_i = \lm (h_i)g_i$ and $\mu _i = \lm (h_i)$   for those indices  $i$  with  $\lm (h_i g_i) = \mu$,  
we shall have an alternative equality  $f = \sum _{i=1} ^{r'} h'_i g'_i$  such that  the maximum monomial  $\mu '$  is smaller than  $\mu$, and the proof will be through.

\vspace{6pt}\noindent
{\it Claim 1 :  Assume that  $f_1,f_2, \ldots , f_s \in S$  are polynomials in  $S$  having the same leading monomial  $\mu$. 
If  $\lm  (\sum _{j=1}^{s} c_j f_j) < \mu$  holds for  some  $c_j \in k$, 
 then $\sum _{j=1}^{s} c_j f_j$  is described as a linear combination of the $S$-polynomials  $S(f_j,f_{\ell})  \  \ (1 \leq j < \ell \leq s)$.}
\vspace{6pt}

\vspace{6pt}\noindent
{\it Claim 2 :
Let  $f_{i_1} , f_{i_2}$   be non-zero elements of  $S$  and let  $\mu_{i_1}, \mu_{i_2}$  be monomials. 
If  $\lm (\mu _{i_1} f_{i_1}) = \lm (\mu_{i_2} f_{i_2})$, then 
 $S(\mu _{i_1} f_{i_1}, \mu _{i_2} f_{i_2}) = \mu _{i_1} f_{i_1} - \mu _{i_2} f_{i_2}$  is a multiple of  $S(f_{i_1}, f_{i_2})$.
}

\vspace{6pt}
\noindent
Proofs of the claims are similar to the ordinary cases and we leave them to the reader.
\end{proof}

In the following proposition, we assume that  $S$  is a graded ring and that the monomial order is homogeneous.

\begin{proposition}\label{Macaulay}
Let $I$ be a homogeneous ideal of $S$, and let $\G$  be a Gr\"obner base for $I$ consisting of homogeneous polynomials. 
Define  $\varphi : S \to S$  by mapping $f \in S$  to the remainder of  $f$  with respect to  $\G$. 
Then $\varphi$  induces a mapping $\overline{\varphi} : S/I \to S/in (I)$ which is an isomorphism as graded $k$-vector spaces. 
\end{proposition}

\begin{proof}
Since  $\G$  is a Gr\"obner base, $\varphi (f)$  is uniquely determined for $f \in S$ by  Proposition \ref{division}. 
If  $\varphi (f) \in in (I)$, then $\varphi (f) =0$  and it follows from Corollary \ref{membership}  that $f \in I$. 
Hence  $\overline{\varphi}$  is an injection.
For any monomial $\mu \not\in in (I)$, it is clear that  $\mu$ itself is a remainder of  $\mu$  with respect to  $\G$. 
Therefore  $\overline{\varphi}$  is surjective. 
\end{proof}

Using the mapping  $\overline{\varphi}$, one can construct a one-to-one correspondence between the $k$-bases of $S/I$  and  $S/in (I)$. 
This idea is a key for the argument in the next section. 

Let  $R$  be a residue ring of the graded polynomial ring   $S$   by a homogeneous ideal  $I$, i.e.  $R = S/I$.
Then  $R$  is also a graded ring  and we denote by  $R_n$  the part of degree $n$  of  $R$. 
Recall that the Hilbert series of  $R$  is defined to be 
$$
H_{R} (T) = \sum _{i=0}^{\infty} \ ( \Kdim _k R_n ) \ T^n  
$$
which is an element of  $\Z [[ T]]$. 
Note that 
$$
H_S (T) = \prod _{i=1}^{\infty} \frac{1}{1-T^{d_i}} 
$$
where  $d _i = \deg (x_i)$. 
Remark that by the definition of graded structure of  $S$, each  $S_n$, hence $R_n$,   is of finite dimension over  $k$. 
In particular, there are only a finite number of variables  $x_i$  with  $\deg (x_i) \leq n$  for each integer  $n$. 
Hence  $H_R (T)$ and  $H_S(T)$  are well-defined elements of  $\Z [[T]]$.

Proposition \ref{Macaulay}  implies the following equality for Hilbert series.

\begin{corollary}
Under the same assumption as the proposition, the equality 
$
H_{S/I} (T) = H_{S/in (I)} (T) 
$
holds.   
\end{corollary}


Now we are discussing about the regularity condition for a sequence of elements in the ring.

\begin{definition}
Let $R$  be an arbitrary ring in this definition.  
Let  $\Omega$  be a well-ordered set and suppose we are given a set of elements   $\SS = \{ f_{\alpha} \in R  \ | \ \alpha \in \Omega \}$  indexed by $\Omega$.

\begin{itemize}
\item[(1)]
We call  $\SS$  a regular sequence on  $R$  if  
 $R/(f_{\alpha} \ | \ \alpha \in \Omega )R \not= 0$  and   
$f_{\alpha}$  is  a non-zero divisor on  $R/(f_{\beta} \ | \ \beta < \alpha)R$  for all $\alpha \in \Omega$. 
If  $R$  is a graded ring, then a regular sequence consisting of homogeneous elements in  $R$  is called a homogeneous regular sequence.

\item[(2)]
We say that the sequence $\SS$ satisfies the FR-condition (the finite regularity condition), if any finite subsequences $\{ f_{\alpha_1},f_{\alpha_2}, \ldots ,f_{\alpha_r}\}$  with $\alpha_1 < \alpha_2 < \ldots < \alpha_r$  in  $\Omega$ is a regular sequence in this order. 
\end{itemize}
\end{definition}

We shall prove that these two regularity conditions above are equivalent for homogeneous sequences. 
One implication holds for any sequence and it is easily proved as in the following lemma.

\begin{lemma}\label{0}
If an ordered set  $\{ f_{\alpha} \in R \ | \ \alpha \in \Omega \}$ satisfies the FR-condition, then it is a regular sequence.
\end{lemma}

\begin{proof}
If  $R = (f_{\alpha} \ | \ \alpha \in \Omega )R$, then  
we have the equality  $1 = \sum _{i=1}^r f_{\alpha _i}h_i$  for some  $f_{\alpha_1},\ldots ,f_{\alpha_r}$ and  $h_i \in R$. 
Thus $R =  (f_{\alpha_1}, \ldots ,f_{\alpha_r})R$ and this contradicts to the FR-condition. 
Therefore  we have  $R/(f_{\alpha} \ | \ \alpha \in \Omega )R \not= 0$.

Suppose that $hf_{\alpha} \in (f_{\beta} \ | \ \beta < \alpha)R$ for $h \in R$. 
Then we have the expression  $h f_{\alpha} = \sum_{i=1}^{r} f_{\beta_i} h_i$  for some  $\beta _1 < \cdots < \beta _r < \alpha$  and  $h_i \in R$.
Since $\{ f_{\beta_1}, \ldots ,f_{\beta_r} ,f_{\alpha} \}$ is a regular sequence in this order, we get 
$h \in (f_{\beta_1},\ldots , f_{\beta_r})R \subset (f_{\beta} \ | \ \beta < \alpha)R$.
\end{proof}

To prove the other implication we need several lemmas.

\begin{lemma}\label{1}
Let $\{ f_{\alpha} \in R  \ | \ \alpha \in \Omega \}$ be a regular sequence indexed by a well-ordered set  $\Omega$, and set  $I = ( f_{\alpha} \ | \ \alpha \in \Omega)R$. 
Let $\Bbb Y = \{ Y_{\alpha} \ | \ \alpha \in \Omega \}$ be a set of indeterminates over $R$  corresponding to $\Omega$. 
For a homogeneous polynomial $F \in R[\Bbb Y]_d$ of degree $d$  in  $\Bbb Y$, we denote by  $F(f)$  the elements of  $R$  obtained by substituting  $f_{\alpha}$  for  $Y_{\alpha}$. 
Under this notation, if $F(f) \in I^{d+1}$, then $F \in IR[\Bbb Y]$.
\end{lemma}

\begin{proof}
Using a transfinite induction on $\alpha \in \Omega$ we shall prove a more strong statement: 

\vspace{6pt}\noindent
{\it Claim 1: 
 Let  $I_{\alpha}  = ( f_{\beta} \ | \ \beta \leq \alpha)R$ and $\Bbb Y _{\alpha} = \{ Y_{\beta} \ |\  \beta \leq \alpha \}$. 
For a homogeneous polynomial  $F \in R[\Bbb Y_{\alpha}]_d$, if $F(f) \in I_{\alpha}^{d+1}$, then $F \in I_{\alpha} R[\Bbb Y_{\alpha}]$.
}\vspace{6pt}

\noindent
As the transfinite induction hypothesis, we assume that Claim 1 holds for any  $\alpha ' \in \Omega$  with  $\alpha ' < \alpha$. 
To prove Claim 1 we need the following auxiliary result.

\vspace{6pt}\noindent
{\it Claim 2: 
 Let  $J _{\alpha}= ( f_{\beta} \ | \ \beta < \alpha)R$ and $\Bbb Y_{\alpha}' = \{ Y_{\beta} \ |\  \beta < \alpha \}$. 
Then  $f_{\alpha}$ is a non-zero divisor on $R/J_{\alpha}^{j}$ for all $j \geq 1$.
}\vspace{6pt}

In fact, suppose that $hf_{\alpha} \in J_{\alpha}^{j}$ for some $j > 1$. 
Arguing by the induction on  $j$  we have $h \in J_{\alpha}^{j-1}$, hence $h = H(f)$  for some  $H \in R[\Bbb Y_{\alpha}']_{j-1}$. 
Thus  $f_{\alpha} H(f) =  f_{\alpha} h \in J_{\alpha}^{(j-1)+1}$, and 
applying the transfinite induction hypothesis to $f_{\alpha} H \in R[\Bbb Y_{\alpha}']_{j-1}$, we have $f_{\alpha} H \in J_{\alpha}R[\Bbb Y_{\alpha}']$. 
Since $f_{\alpha}$ is a non-zero divisor on $R/J_{\alpha}$, we have $H \in J_{\alpha}R[\Bbb Y_{\alpha}']_{j-1}$, therefore $h = H(f) \in J_{\alpha}^{j}$.

Now we proceed to the proof of Claim 1. 
For this, let  $F \in R[\Bbb Y_{\alpha}]_d$. 
We shall prove by the induction on  $d$  that  $F (f) \in I_{\alpha}^{d+1}$ 
implies   $F \in I_{\alpha}R[\Bbb Y_{\alpha}]$. 
If $d = 0$ then the claim is trivial, hence we assume that $d > 0$. 

First we show that we may assume that $F(f) = 0$. 
Since $F(f) \in I_{\alpha}^{d+1}$, there exists  $G \in R[\Bbb Y_{\alpha}]_{d+1}$ such that $F(f) = G(f)$. 
Writing  $G = \sum_{i=1}^{n} Y_{\beta _{i}} G_{i}$ with  $\beta _i \leq \alpha$  and  $G_{i} \in R[\Bbb Y_{\alpha}]_{d}$, 
we set $G' = \sum_{i=1}^{n} f_{\beta _{i}} G_{i}$. 
Then we have $F - G' \in R[\Bbb Y_{\alpha}]_{d}$ and $(F - G')(f) = 0$. 
Furthermore, it holds that  $F - G' \in I_{\alpha}R[\Bbb Y_{\alpha}]$ if and only if $F \in I_{\alpha}R[\Bbb Y_{\alpha}]$. 

Henceforth we assume $F(f) = 0$. 
Then we may write $F = G + Y_{\alpha} H$ with $G \in R[\Bbb Y_{\alpha}']_d$  and  $H \in R[\Bbb Y_{\alpha}]_{d-1}$. 
Since  $f_{\alpha} H(f) = -G(f) \in J _{\alpha}^d$, the Claim 2 above implies that $H(f) \in J_{\alpha}^{d} \subset I_{\alpha}^{(d-1)+1}$. 
Thus, by the induction on $d$, we get $H \in I_{\alpha}R[\Bbb Y_{\alpha}]$. 
On the other hand, since  $H(f) \in  J_{\alpha}^{d}$, there is 
 $H' \in R[\Bbb Y_{\alpha}']_{d}$  such that $H(f) = H'(f)$ holds. 
As
$
(G + f_{\alpha}H')(f) = F(f) = 0,
$
it follows by transfinite induction on $\alpha$ that $G + f_{\alpha}H' \in J_{\alpha}R[\Bbb Y_{\alpha}'] \subset I_{\alpha}R[\Bbb Y_{\alpha}]$. 
Since $f_{\alpha} H' \in I_{\alpha}R[\Bbb Y_{\alpha}]$, we get $G \in I_{\alpha}R[\Bbb Y_{\alpha}]$. 
Therefore $F \in I_{\alpha}R[\Bbb Y_{\alpha}]$.
This completes the proof. 
\end{proof}

\begin{corollary}\label{2}
Let $\{ f_{\alpha} \in R  \ | \ \alpha \in \Omega \}$ be a regular sequence indexed by a well-ordered set  $\Omega$, and set  $I = ( f_{\alpha} \ | \ \alpha \in \Omega)R$. 
And let $\Bbb Y = \{ Y_{\alpha} \ | \ \alpha \in \Omega \}$ be a set of indeterminates over $R$  corresponding to $\Omega$ as in the lemma. 
Then the map $\varphi :  (R/I)[\Bbb Y] \to gr_{I}(R) = \bigoplus_{n=0}^{\infty} I^{n}/I^{n+1}$ induced by the substitution $Y_{\alpha} \mapsto \bar f_{\alpha} \in I/I^{2}$ is an isomorphism as algebras over  $R/I$. 
In particular, $I^{n}/I^{n+1}$ is an $(R/I)$-$free$ $module$ for all $n \in \Bbb N$. 
\end{corollary}

\begin{proof}
From the definition,  $\varphi$  is a well-defined algebra map over  $R/I$  that is surjective. 
It follows from Lemma \ref{1} that  $\varphi$  is injective as well.
\end{proof}

\begin{lemma}\label{3}
Let  $R = \bigoplus_{n=0}^{\infty} R_{n}$ be a non-negatively graded ring, and let  $\{ f_{\alpha}\ |\  \alpha \in \Omega \}$ be a sequence of homogeneous elements of positive degree in  $R$  indexed by a well-ordered set  $\Omega$. 
If $\{ f_{\alpha} \ | \ \alpha \in \Omega\}$ satisfies the FR-condition, then so does any permutation. 
More precisely, if $\Omega '$  is another well-ordered set such that there is a  bijective mapping  $\sigma : \Omega '  \to \Omega$, then the sequence 
 $\{ f_{\sigma ( \alpha' ) }\ |\  \alpha ' \in \Omega ' \}$  satisfies the FR-condition whenever  
 $\{ f_{\alpha}\ |\  \alpha \in \Omega \}$ does. 
\end{lemma}

\begin{proof}
By the definition of the FR-condition it is enough to show that 
any permutation of a finite homogeneous regular sequence $f_1,f_2,\ldots,f_r$ is again a regular sequence. 
For this we have only to show that if $f,g$ is a homogeneous regular sequence,  then so is $g,f$. 
Suppose that  $f, g$  is a homogeneous regular sequence on  $R$. 
Assume $hg = 0$ for $h \in R$ and we want to show that  $h=0$. 
We may assume that $h$ is homogeneous. 
Since $hg \in (f)R$, we see $h \in (f)R$. 
Writing  $h = h_{1}f$  for some  $h_1 \in R$, we have $h_{1}g = 0$ since $f$ is a non-zero divisor on $R$. 
Hence  $h_1 \in (f)R$, and $h \in (f^2)R$. 
Subsequently we have  $h \in (f^n)R$ for any  $n \geq 1$.
Take $n$ so that  $n \deg f > \deg h$ and we conclude that  $h=0$. 

Next assume $h'f \in (g)R$ for $h' \in R$. 
Writing  $h'f = gh'_{1}$, we have $h'_{1} = f h'_{2}$ for some  $h_2 \in R$,  since  $gh'_{1} \in (f)R$ and  $g$  is a non-zero divisor on  $R/(f)R$. 
Since  $f$  is a non-zero divisor on  $R$, we have  $h' = gh'_2 \in (g)R$ as desired. 
\end{proof}

\begin{proposition}\label{4}
Let  $R = \bigoplus_{n=0}^{\infty} R_{n}$ be a non-negatively graded ring, and let  $\{ f_{\alpha}\ |\  \alpha \in \Omega \}$ be a sequence of homogeneous elements of positive degree in  $R$  indexed by a well-ordered set  $\Omega$. 
If $\{ f_{\alpha} \ | \ \alpha \in \Omega\}$ is a regular sequence on $R$, then   it satisfies the FR-condition. 
\end{proposition}

\begin{proof}
Suppose there is a finite sequence $\{ f_{\alpha_{1}}, f_{\alpha_{2}}, \ldots, f_{\alpha_{r}} \}$ with  $\alpha_{1} < \alpha_{2} < \ldots < \alpha_{r}$
such that it is not a regular sequence.
Take such a $\{ f_{\alpha_{1}}, f_{\alpha_{2}}, \ldots, f_{\alpha_{r}} \}$ with $\alpha_{r}$ being minimum in $\Omega$. 
Note by this choice of  $\alpha _r$ that $\{ f_{\beta}\ | \ \beta < \alpha_{r} \}$ satisfies the FR-condition. 
After changing the order of first $(r-1)$ elements in the sequence we may assume the following:

\begin{itemize}
\item[(i)] $\alpha_{1},\ldots,\alpha_{r-1}$ are the first $(r-1)$ elements in $\Omega$, 
\item[(ii)] $\{ f_{\beta} \ |\ \beta < \alpha_{r} \}$ satisfies the FR-condition, 
\item[(iii)] $\{ f_{\beta} \ |\  \beta < \alpha_{r} \} \sqcup \{ f_{\alpha_{r}} \}$ is a regular sequence, 
\item[(iv)] $\{ f_{\alpha_{1}}, f_{\alpha_{2}}, \ldots, f_{\alpha_{r}} \}$ is not a regular sequence. 
\end{itemize}

Considering the residue ring  $\bar R = R/(f_{\alpha_{1}}, \ldots , f_{\alpha_{r-1}})R$, we have: 

\begin{itemize}
\item[(i)']  $\{ \bar f_{\beta} \in \bar R \ |\  \alpha_{r-1} < \beta < \alpha_{r} \}$ satisfies the FR-condition, 
\item[(ii)']  $\{ \bar f_{\beta} \in \bar R \ | \ \alpha_{r-1} < \beta < \alpha_{r} \} \sqcup \{ \bar f_{\alpha_{r}} \}$ is a regular sequence on $\bar R$, 
\item[(iii)']  $\bar f_{\alpha_{r}}$ is a zero divisor on $\bar R$. 
\end{itemize}

We show a contradiction from this setting. 
Set $J = (\bar f_{\beta} \ | \ \alpha_{r-1} < \beta < \alpha_{r}) \bar R$. 
By (iii)'  there is a non-zero homogeneous element  $\bar g \in \bar R$
such that  $\bar g \bar f_{\alpha_{r}} = 0$. 
Then we can find  $n \in \Bbb N$ such that $\bar g \in J^{n} \setminus J^{n+1}$. 
Since $\bar f_{\alpha_{r}}$ is a non-zero divisor on $\bar R/J$ by (ii)', and since $\bar f_{\alpha_{r}} \bar g \equiv 0$ in $J^{n}/J^{n+1}$, we apply Corollary \ref{2} using  (i)' and we conclude that $\bar g \in J^{n+1}$. 
This is a contradiction. 
\end{proof}

As a consequence of \ref{0}, \ref{3} and \ref{4} we have the following result.

\begin{theorem}\label{5}
Let  $R = \bigoplus_{n=0}^{\infty} R_{n}$ be a non-negatively graded ring, and let  $\SS = \{ f_{\alpha}\ |\  \alpha \in \Omega \}$ be a sequence of homogeneous elements of positive degree in  $R$  indexed by a well-ordered set  $\Omega$. 
Then $\SS$ is a regular sequence on $R$ if and only if 
$\SS$  satisfies the FR-condition.  
In particular,  any permutation of a homogeneous regular sequence is again a regular sequence. 
\end{theorem}


Now we return to the case for the polynomial ring  $S=k[x_1, x_2, \ldots ]$.

\begin{lemma}\label{regular}
\begin{itemize}
\item[$(1)$]
Let   $\SS$   be a set of monomials in  $S$. 
Then  $\SS$  is a homogeneous regular sequence on  $S$  if and only if 
any two monomials in  $\SS$  are coprime, 
i.e.  any distinct elements  $\mu_1$  and  $\mu _2$  in  $\SS$  have no common divisor except units. 

\item[$(2)$]
Let  $\SS$  be a homogeneous regular sequence on $S$ and let  $n$  be an integer. 
\begin{itemize}
\item[$(a)$]
Then  $\SS \cap \Sn$  is a finite set consisting of at most  $n$ elements. 
\item[$(b)$]
The set  $\{ f \in \SS \ | \ \deg (f) \leq n \}$  is a finite set. 
\item[$(c)$]
$\SS$  is a countable set. 
\end{itemize}
\end{itemize}
\end{lemma}

\begin{proof}
It is easy to prove (1) and we leave it the reader. 
To prove (2)(a), note that  $\SS \cap \Sn$  is a homogeneous regular sequence on $\Sn$, since  $\Sn \subset S$  is a faithfully flat ring extension. 
The graded $k$-algebra  $\Sn$  has depth $n$, hence any homogeneous regular sequence on  $\Sn$  has at most length  $n$. 
It forces $| \SS \cap \Sn | \leqq n$. 
For  (2)(b), recall that the grading for  $S$  is given in such a way that  there are only a finite number of monomials of degree  $n$  for each  integer  $n$.  Therefore,  for any integer  $n > 0$,  there is an integer  $m$   such that the equality  $S_n = (\Sm) _n$  holds. 
Therefore the set  $\{ f \in \SS \ | \ \deg (f) \leq n\}$  is a subset of  $\SS \cap \Sm$ that is a finite set.
Since the equality  $\SS = \bigcup _{n=0}^{\infty} \{ f \in \SS \ | \ \deg (f) \leq n\}$ holds,  $\SS$  is a countable set. 
\end{proof}

\begin{proposition}
Let  $\SS$  be a homogeneous regular sequence on $S$, and let  $I$  be the homogeneous ideal generated by  $\SS$. 
Then the Hilbert series of  the graded ring  $S/ I$  is given by
$$
H_{S/I} (T) = H_S(T) \cdot \prod _{f \in \SS} (1-T^{\deg (f)}).  
$$
\end{proposition}

\begin{proof}
For a graded $S$-module $M = \bigoplus _{i=0}^{\infty} M_i$,  we denote by  $M_{\leq n}$  the $k$-subspace of $M$  spanned by homogeneous elements of degree at most $n$ ; 
$
M_{\leq n} = \bigoplus _{i \leq n} M_i.  
$
Now let  $\{ f_1 , \ldots, f_r\}$  be the set of all the elements of $\SS$  of degree at most $n$ which is a finite set by Lemma \ref{regular} (2). 
Then it is easy to see that  $I _{\leq n} = ((f_1, \ldots , f_r )S)_{\leq n}$. 
Thus  $(S/I) _{\leq n} = ( S /(f_1, \ldots , f_r )S ) _{\leq n}$.   
From the definition of graded structure of  $S$, given an integer $n$, we have an integer $m$  with  the equality    $S_{\leq n} = ( \Sm ) _{\leq n}$.  
We can take such an integer $m$ as $\Sm$  contains  $f_1, \ldots , f_r$. 
Therefore we have $(S/I) _{\leq n} = ( \Sm /(f_1, \ldots , f_r )\Sm ) _{\leq n}$. 
This implies that the difference  
$H_{S/I}(T) - H_{\Sm /(f_1, \ldots , f_r )\Sm} (T)$  belongs to  $T^n \Z [[T]]$. 
Note that the equality 
$$
H_{\Sm /(f_1, \ldots , f_r )\Sm} (T) = H_{\Sm}(T) \cdot \prod _{i=1} ^r (1-T^{\deg (f_i)})
$$ 
is known to hold by \cite[Exercise 19.14]{E}, since  $\{ f_1 , \ldots, f_r\}$ is a regular sequence on  $\Sm$.
Thus we have 
$$
H_{S/I}(T) - H_{S}(T) \cdot \prod _{f \in \SS, \deg (f) \leq n} (1-T^{\deg (f)})  \in T^n \Z [[T]].
$$
Since this holds for any integer $n$, the proof is completed. 
\end{proof}

The criterion of Bayer-Stillman \cite[Proposition 15.15]{E}  is generalized in the following form.

\begin{proposition}\label{bayer}
Let  $\SS = \{ f_{\alpha} \ | \ \alpha \in \Omega \}$  be a set of elements in  $S$. 
Assume that  $\{ \lm (f_{\alpha}) \ | \ \alpha \in \Omega \}$  is a regular sequence on $S$. 
Then  $\SS$  is a regular sequence on $S$  and it is a Gr\"obner base for the ideal  $\SS S$. 
\end{proposition}

\begin{proof}
To prove that  $\SS$  is a Gr\"obner base, we have only to show that $0$  is a remainder of  $S (f, g)$ with respect to  $\SS$  for any $f, g \in \SS$. 
See Proposition \ref{buchberger}. 
More strongly we can show that $0$  is a remainder of 
$S(f, g)$  with respect to  $\{ f, g\}$,  whenever  $\lm (f), \lm (g)$  is a regular sequence on  $S$. 
In fact,  $S(f, g) = \lt (g)f -  \lt (f)g = -(g - \lt (g))f + (f - \lt (f))g$  holds and it is easy to see that this description shows that the remainder is $0$. 

Now to prove that  $\SS$  is a regular sequence on  $S$, 
we assume that $h f_{\alpha} \in (f_{\beta} \ | \ \beta < \alpha)S$  for  $h \in S$  and  $\alpha \in \Omega$. 
We shall show that  $h \in (f_{\beta} \ | \ \beta < \alpha)S$  by the induction  on  $\lm (h)$. 
Then we have  $\lm (h) \lm (f_{\alpha}) \in in ((f_{\beta} | \beta < \alpha)S)$. 
Since  $\{ \lm (f_{\beta}) \ | \ \beta < \alpha \}$ is a regular sequence on $S$, it follows from the first half of the proof that  $\{ f_{\beta} \ | \ \beta < \alpha \}$ is a Gr\"obner base for the ideal $( f_{\beta} \ | \ \beta < \alpha)S$. 
Thus there is a monomial  $\lm (f_{\beta_1})$ with  $\beta _1 < \alpha$  which divides  $\lm (h)$. 
Therefore  $\lm (h - c_1 \mu_1 f_{\beta_1}) < \lm (h)$ holds for some $c_1 \in k$  and  a monomial  $\mu_1$. 
Then it follows from the induction hypothesis that  $h - c_1 \mu _1 f_{\beta_1} \in (f_{\beta} \ | \ \beta < \alpha)S$, hence 
$h \in  (f_{\beta} \ | \ \beta < \alpha)S$. 
\end{proof}

\section{Applications}

Let $S = k[x_1, x_2, \ldots]$ be a polynomial ring with countably infinite variables as before.
We regard $S$ as a graded $k$-algebra by defining $\deg (x_i) = i$ for each $i \in \N$, and  denote by  $S_n$  the part of degree $n$  of  $S$ for  $n \in \N$.
Note that there is a bijective mapping between the set of partitions of $n$ and the set of monomials of degree $n$. 
The correspondence is given by mapping a partition $\lambda = (\lambda _1, \lambda _2, \ldots , \lambda _r) \vdash n$  to the monomial $x^{\lambda} = x_{\lambda _r} \cdots x_{\lambda _2}x_{\lambda _1}$  of degree $n$.

\vspace{6pt}

Let  $W$  be any subset of $\N$ satisfying  $pW \subset W$ for an integer $p \ge 2$,  where  $pW = \{ pw \ | \ w \in W\}$. 
In this case, we consider a subring  $R = k[X_i \ | \ i \in W]$  of $S$.
We are interested in the following two subsets of partitions of $n$: 
$$
\begin{array}{rcl}
X(n) &= \{ \lambda \vdash n &|\ \ \lambda _i \in W \setminus pW \}, \vspace{4pt}\\
Y(n) &= \{ \lambda \vdash n &|\ \ \lambda _i \in W , \ \text{and  any  number  appears among  the $\lambda _i$'s} \\
& &\ \ \ \text{at most $p-1$ times} \}. 
\end{array}
$$

\begin{theorem}\label{main theorem}
Under the circumstances above, consider the set of homogeneous polynomials $\G = \{ x_{i}^{p} - x_{pi} \ | \ i \in W \}$  in   $R$. 
We adopt the homogeneous anti-reverse lexicographic order (resp. the homogeneous lexicographic order)  on the set of monomials in  $R$. 
Then $\G$  is a reduced Gr\"obner base (resp. a Gr\"obner base)  for the ideal $\G R$. 

Furthermore, define a mapping  $\varphi : X(n) \to Y(n)$  so that   
$x^{\varphi(\lambda)}$  is the remainder of  $x^{\lambda}$ with respect to  $\G$ in the homogeneous anti-reverse lexicographic order for any  $\lambda \in X(n)$. 
Then $\varphi$ is a well-defined bijective mapping. 
\end{theorem}

\begin{proof}
Note that  $x_i^p >_{harl} x_{pi}$ (resp. $x_{pi} >_{hl} x_i^p$), hence we have  $\lm (x_{i}^{p}-x_{pi}) = x_{i}^{p}$ in homogeneous anti-reverse lexicographic order  (resp. $\lm (x_{i}^{p}- x_{pi}) = x_{pi}$ in homogeneous lexicographic order ) for all $i \in W$. 
Since it is clear that  $\{  x_i ^p \ | \ i \in W\}$  (resp. $\{ x_{pi} \ | \ i \in W\}$)  is a homogeneous regular sequence on  $R$, it follows from Proposition \ref{bayer} that 
 $\{ x_{i}^{p} - x_{pi} \ | \ i \in W \}$ is a homogeneous regular sequence on $R$, which is a Gr\"obner base. 
Actually this is a reduced Gr\"obner base in the case of  homogeneous anti-reverse lexicographic order.

To prove the second half of the theorem, 
let  $\lambda \in X(n)$ be an arbitrary element. 
By definition  $x^{\lambda}$ contains no variables  $x_{pi} \ (i \in W)$.
In order to get the remainder of  $x^{\lambda}$  with respect to  $\G$  in the homogeneous anti-reverse lexicographic order, 
we replace  $x_i^p$  with  $x_{pi}$  in the monomial, 
 whenever  $x^{\lambda}$  involves a $p$th power  $x_i^p$  of a variable. 
 Continue this procedure until we get the monomial $x^{\rho}$  involving no $p$th power of a variable.   
It is then clear that  $\rho \in Y(n)$  and  $x^{\rho}$  is the remainder of  $x^{\lambda}$  with respect to  $\G$  in the homogeneous anti-reverse lexicographic order. 
In  such a way we have  $\varphi (\lambda) = \rho$, hence the mapping  $\varphi : X(n) \to Y(n)$ is well-defined. 

In a similar manner to this, we can define  $\psi : Y(n) \to X(n)$ by using the homogeneous lexicographic order and by replacing  $x_{pi}$  with  $x_i^p$  in the monomials, and it is obvious by the construction that 
$\varphi \cdot \psi =id _{Y(n)}$  and  $\psi \cdot \varphi = id _{X(n)}$. 
\end{proof}

Just considering the generating functions of $|X(n)|$ and $|Y(n)|$, 
we see that the following equality holds;
$$
\prod_{m \in W \setminus pW} \frac{1}{1-t^m} = \prod_{m \in W} (1 + t^m + t^{2m} + \cdots + t^{(p-1)m}).
$$

\begin{example}
Recall that  $A(n)$, $B(n)$ and $C(n)$  are the sets of partitions given in Introduction.
\begin{itemize}
\item[$(1)$] 
If  $W = \{ n \in \N \ |\ n \equiv \pm 1 \pmod 3 \}$ and $p=2$, then  $X(n) = A(n)$ and $Y(n) = B(n)$. 
\item[$(2)$]
If  $W = \{ n \in \N \ |\ n \equiv 1 \pmod 2 \}$ and $p=3$,  then $X(n) = A(n)$ and $Y(n) = C(n)$. 
\end{itemize}
\end{example}

As a consequence of all the above, we obtain one-to-one correspondences among $A(n)$, $B(n)$ and $C(n)$ by using the theory of Gr\"obner bases. 
Considering their generating functions we have the following equalities: 
$$
\prod_{m \equiv \pm 1 \pmod 6} \frac{1}{1-t^m} 
= \prod_{m \equiv \pm 1 \pmod 3} (1 + t^m), 
= \prod_{m \equiv 1 \pmod 2} (1 + t^m + t^{2m})
$$
which are called the Schur's equalities.
See \cite[(1.2) and (1.3)]{A}.

\vspace{6pt}

We close the paper by raising a problem. 
For this let us consider the following sets of partitions. 
$$
\begin{array}{rcl}
P(n) &= \ \{ \ \lambda \vdash n  &| \ \ \lambda _i \equiv \pm 1 \pmod 5 \ \}, \vspace{4pt} \\
Q(n) &= \ \{ \ \lambda \vdash n  &| \ \ \lambda _i - \lambda _{i+1} \geq 2 \ \}. \end{array}
$$
By Rogers-Ramanujan equality  
$$
\prod_{m \equiv \pm1 \pmod 5} \frac{1}{1-t^m} = 1 + \sum_{m=1}^{\infty} \frac{t^{m^2}}{(1-t)(1-t^2) \cdots (1-t^m)},
$$
it is known that the sets  $P(n)$  and  $Q(n)$  have the same cardinality for each $n \in \N$. 
(See \cite[(5.26)]{B}.)

If we find an ideal $I$  as in the following problem, then we will obtain a one-to-one correspondence between   $P(n)$  and  $Q(n)$  by using division algorithm.

\begin{problem}
Find an ideal $I$ of  $S$  and a monomial order on  $\Mon (S)$  satisfying  $S/I \cong k[\{x_i \ | \ i \equiv \pm1 \pmod 5 \} ]$  and  $in (I) = (x_i^2 , x_ix_{i+1} \ |\  i \in \N)$.
\end{problem}


\vspace{12pt}

\begin{center}
Kei-ichiro Iima; Graduate School of Natural Science and Technology, Okayama University, Okayama 700-8530, Japan
\texttt{(e-mail:iima@math.okayama-u.ac.jp)} \\
Yuji Yoshino; Department of Math., Okayama University, Okayama 700-8530, Japan 
\texttt{(e-mail:yoshino@math.okayama-u.ac.jp)}
\end{center}

\end{document}